\documentclass[10pt, a4paper]{amsart}
\calclayout
\usepackage[utf8]{inputenc}
\calclayout
\usepackage[T1]{fontenc}
\usepackage{color}
\usepackage[english]{babel}
\usepackage[mathscr]{euscript}
\usepackage[mathscr]{euscript}
\usepackage{amsthm}
\usepackage{thmtools}
\usepackage{stmaryrd}
\usepackage[utf8]{inputenc}
\usepackage[all]{xy}
\usepackage{indentfirst}
\usepackage{hyperref}
\usepackage{amsmath}
\usepackage{amsmath,amsfonts}
\usepackage{lipsum}
\usepackage[toc,page]{appendix}
\usepackage[pdftex]{graphicx}
\usepackage{tikz-cd}
\usepackage{amsthm}
\usepackage{stmaryrd}
\usepackage{ragged2e}
\usepackage{blindtext}
\usepackage{url}
\usepackage{booktabs} 
\usepackage{array} 
\usepackage{paralist} 
\usepackage{verbatim} 
\usepackage{subfig}
\usepackage{graphicx}			
\input xy
\xyoption{all}

\newtheorem{Theorem}{Theorem}[section]
\newtheorem{Lemma}[Theorem]{Lemma}
\newtheorem{Proposition}[Theorem]{Proposition}
\newtheorem{Definition}[Theorem]{Definition}
\newtheorem{Remark}[Theorem]{Remark}

\newtheorem{Corollary}[Theorem]{Corollary}
\newtheorem{Example}[Theorem]{Example}

\DeclareMathOperator{\Ker}{Ker}
\DeclareMathOperator{\ver}{ver}
\DeclareMathOperator{\hor}{hor}

\DeclareMathOperator{\unm}{unm}

\DeclareMathOperator{\R}{R}

\DeclareMathOperator{\Fitt}{Fitt}
\DeclareMathOperator{\SD}{SD}

\DeclareMathOperator{\Tot}{Tot}

\DeclareMathOperator{\Kitt}{Kitt}

\DeclareMathOperator{\V}{V}

\DeclareMathOperator{\RI}{RI}

\DeclareMathOperator{\coker}{Coker}
\DeclareMathOperator{\htt}{ht}
\DeclareMathOperator{\Ht}{ht}
\DeclareMathOperator{\depth}{depth}

\DeclareMathOperator{\grade}{grade}
\DeclareMathAlphabet{\mathpzc}{OT1}{pzc}{m}{it}
\DeclareMathOperator{\Spec}{Spec}

\DeclareMathOperator{\SDC}{SDC}

\newcommand{\ff}{\mathbf{f}}
\newcommand{\fa}{\mathfrak{a}}

\newcommand{\x}{{\textbf{x}}}

\newcommand{\fm}{\mathfrak{m}}

\newcommand{\fc}{\frak{c}}
\newcommand{\fn}{\frak{n}}
\newcommand{\fp}{\frak{p}}
\newcommand{\fq}{\frak{q}}
\newcommand{\fg}{\frak{g}}

\newcommand{\fb}{\frak{b}}

\begin{document}
\title{Deformation of Residual Intersections}
\author{Hamid Hassanzadeh and  Kevin Vasconcellos }
\email{hamid@im.ufrj.br}
\email{kevin.vasconcellos@uerj.br}
\address{Departamento de Matem\'atica, Centro de Tecnologia, Cidade Universit\'aria da Universidade Federal do Rio de Janeiro, 21941-909 Rio de Janeiro, RJ, Brazil}
\address{Departamento de Matem\'atica, Universidade do Estado do Rio de Janeiro, Rua S\~ao Francisco Xavier, 524 - 6 andar - Maracan\~a, Rio de Janeiro, RJ, Brazil}
\thanks{The first was partially supported by grants CAPES-PrInt Project 88881.311616/2018-00, Finance Code 001, FAPERJ 2025, APQ1  and CNPq-Universal No. 406377/2021-9.}
\date{\today}

\maketitle
\begin{abstract} It is shown that in a Cohen-Macaulay local ring, the generic linkage of an ideal $I$ is a deformation of the arbitrary linkage of $I$. This fact does not need $I$ to be a Cohen-Macaulay ideal.  The same holds for $s$-residual intersections of $I$ when $s\leq \Ht(I)+1$. Under some slight conditions on $I$, one further generalizes this principle to encompass any $s$-residual intersection.
\end{abstract}

\section{Introduction}
One of the main facts connecting linkage theory and deformation theory is Buchweitz's \cite{Buch} result:\vspace{0.1cm}
\begin{quote} Let $P=k\llbracket x_{1},\ldots,x_{n}\rrbracket$ the power series ring and $R=P/I$ a reduced ring. If $I$ is in the linkage class of a complete intersection then $T^{2}(R/k,R)=0$, i.e. $R$ is unobstructed.
\end{quote}\vspace{0.1cm}
Here $T^{i}(R/k,R)$ are upper cotangent functors: $T^{0}(R/k, R)$ is the module of $k$-derivations Der$_{k}(R,R)$, $T^{1}(R/k,R)$ corresponds to the space of isomorphism classes of first-order infinitesimal deformations of $R$ over $k$, and $T^{2}(R/k,R)$ contains the obstructions for lifting infinitesimal deformations of $R$. Therefore $R$ is called rigid over $k $ if $T^{1}(R/k,R)=0$ and nonobstructed if $T^{2}(R/k,R)=0$.

Rigidity is not preserved through an arbitrary linkage class. However Kustin and Miller \cite[Theorem 4.2]{KusMil} show that this property is maintained for ``semi-generic'' linkage. Then it is highly important to see when the generic linkage is a deformation of an arbitrary linkage. 

In their seminal paper \cite{huneke1987structure} Huneke and Ulrich determine the structure of linkage (especially licci ideals) through the study of generic linkage. One of the key tools in this study is the following result, \cite[Proposition 2.14]{huneke1987structure}:\vspace{0.1cm}
\begin{quote}
Let $(R,\fm)$ be a local Gorenstein ring, and $I$ and $J$ two {\bf Cohen-Macaulay} ideals of $R$ which are linked and of positive grade. Considering any first generic linkage $L_{1}(I)$ there exists a prime ideal $\fq \in \Spec(R[\x])$ that contains $\fm$ such that $(R[\x]_{\fq},L_{1}(I)_{\fq})$ is a deformation of $(R,J)$.
\end{quote}
The attention of our work is the Cohen-Macaulay condition in this result.  
While it may seem initially that this condition is not overly restrictive, the evolution of the theory,  in the aforementioned work and subsequent publications like \cite{huneke1988residual} and \cite{huneke1990generic}, led the authors to refine their theory by focusing on Strongly Cohen-Macaulay ideals instead of Cohen-Macaulay ideals.


Surprisingly, we found out that for generic linkage to be a deformation of an arbitrary linked ideal, no conditions must be imposed on $I$. Furthermore, this fact is true for any local Cohen-Macaulay ring $(R,\fm)$ rather than Gorenstein ring in the above setting. In the last result of this work,  \autoref{Lastmain} we show that the same result extends to the class of $(\htt(I)+1)$-residual intersections too.

The achievement is due to the study of residual intersection through {\it Kitt} ideals defined in \cite{boucca2019residual}. It is shown in \autoref{Prop3.1} that the generic Kitt ideals always specialize to the arbitrary Kitt ideal. In \autoref{prop2} we determine when this specialization is via a regular sequence. The question about the deformation of residual intersection then reduces to the cases where Kitt coincides with the residual intersection. To this end, we carefully study the codimension of Kitt ideal in \autoref{heightKitt} which enables us to employ previous results in \cite{hassanzadeh2012cohen} and \cite{boucca2019residual}.  Thence establishes the above surprising fact \autoref{Lastmain}. 

For undefined notations mentioned above, we refer to the Preliminaries section. 
 
 \section{Preliminaries}
In this section, we recall some of our work's basic definitions and notations.  Unexplained notations are taken from the standard books \cite{matsumura1989commutative} and \cite{bruns1998cohen}.

We start by recalling the definitions of $s$-residual intersections.
Let $R$ be a commutative   Noetherian ring and   $I$  an ideal of $R$ with $\grade(I) = g$.  Considering $s\geq g$ an integer, then
\begin{itemize}
\item{An {\it (algebraic) $s$-residual intersection} of $I$ is a proper ideal $J$ of $R$ such that $\grade(J)\geq s$ and $J=(\fa:_RI)$ for some ideal $\fa\subset I$ which is generated by $s$ elements.}
\item{A {\it geometric $s$-residual intersection} of $I$ is an algebraic $s$-residual intersection $J$ of $I$ such that $\grade(I+J)\geq s+1.$}
\item{ An {\it arithmetic $s$-residual intersection} of $I$  is an algebraic $s$-residual intersection such that $\mu_{R_{\fp}}((I/\fa)_{\fp})\leq 1$ for all prime ideal $\fp \supseteq (I+J)$ with $\grade(\fp)\leq s$. ($\mu$ denotes the minimum number of generators.)}
\end{itemize}
The standard definition, cited from \cite{huneke1988residual}, involves substituting the grade of an ideal with its height when $R$ is a Cohen-Macaulay local ring. 

 One of the main tools in our study of residual intersections is {\it Residual Approximation Complexes}. The theory started and developed in \cite{hassanzadeh2012cohen}, \cite{HamidNaeliton}, \cite{ChardinNaeliton} and \cite{boucca2019residual}.

Let $\fa \subseteq I$ finitely generated ideals of $R$. Consider $\mathbf{f}=f_1,\dots,f_r$ and $\mathbf{a}=\:a_1,\dots,a_s$ systems of generators of $I$ and $\fa$, respectively. Let $\Phi=[c_{ij}]$ be an $r\times s$ matrix in $R$ such that
\begin{gather*}
\begin{bmatrix}
a_1 & a_2 & \cdots & a_s
\end{bmatrix}=\begin{bmatrix}
f_1 & f_2 & \cdots & f_r
\end{bmatrix}\begin{bmatrix}
c_{11} & c_{12} & \cdots & c_{1s} \\ 
c_{21} & c_{22}  & \cdots  & c_{2s} \\ 
\vdots & \vdots & \ddots  &\vdots \\ 
c_{r1} & c_{r2} & \cdots & c_{rs}
\end{bmatrix}.
\end{gather*}
Let $S = R[T_1,\dots,T_r]$ be the standard polynomial extension of $R$ in $r$ indeterminates and $$\gamma_j=\sum\limits_{i=1}^rc_{ij}T_i\in S_1$$ for $1\leq j \leq s$. Then we consider the $\mathcal{Z}$-complex, $\mathcal{Z}_\bullet(\ff;R)$, of Herzog-Simis-Vasconcelos \cite{herzog1983koszul}. 
\begin{gather*}
0\rightarrow Z_r(\ff;R)\otimes_RS[-r]\rightarrow Z_{r-1}(\ff;R)\otimes_RS[-r+1]\rightarrow\dots\rightarrow Z_0(\ff;R)\otimes_RS\rightarrow 0,
\end{gather*}
where $Z_i(\ff;R)$ is the $i$-th module of Koszul cycles of the sequence $\ff$. In the sequel we denote the sequence $\underline{\gamma}=\gamma_1,\dots, \gamma_s$ and consider the new complex given by
\begin{gather*}
  \mathcal{D}_\bullet =\Tot\big(\mathcal{Z} _\bullet(\ff;R)\otimes_SK_\bullet(\underline{\gamma},S)\big),  
\end{gather*}
where $K_{\bullet} = K_{\bullet}(\underline{\gamma},S)$ is the Koszul complex of the sequence $\underline{\gamma}$. Notice that the $i$-th component of this complex is given by 
\begin{gather*}
	\mathcal{D}_i = \bigoplus\limits_{k+j=i}(Z_k(\ff;R)\otimes_RS(-k))\otimes_S S^{\binom{s}{j}}(-j) \simeq \bigoplus\limits_{k+j=i}(Z_k(\ff;R)\otimes_RS^{\binom{s}{j}})(-i). 
\end{gather*}
Consider the ideal $\mathfrak{t}=(T_1,\dots,T_r) \subseteq S$. Tensoring the complex $\mathcal{D}_\bullet$ with the \v{C}ech complex $\check{C}^\bullet_{\mathfrak{t}}(S)$, one can proceed as on the construction of the Koszul-\v{C}ech complexes and create a $\mathbb{Z}$-indexed family of complexes $_k\mathcal{Z}^+_\bullet(\ff, \mathbf{a},\Phi)$, where for each $k \in \mathbb{Z}$, one has 
\begin{gather*} _k\mathcal{Z}^+_\bullet(\ff, \mathbf{a},\Phi):\:\:\:\:\:\:\:\:\:
	  0\rightarrow H^r_\mathfrak t(\mathcal D_{r+s})_{(k)}\rightarrow\dots\rightarrow H^r_\mathfrak t (\mathcal{D}_{r+k})_{(k)} \xrightarrow{\tau_k}(\mathcal D_{k})_{(k)}\rightarrow\dots\rightarrow(\mathcal D_0)_{(k)}\rightarrow0.
\end{gather*} 
\noindent Hence we have the following definition.
\begin{Definition}
    Let $R$ be a ring, $\fa\subseteq I$ finitely generated ideals. Consider $\mathbf{f}=\:f_1,\dots,f_r$ and $\mathbf{a}=\:a_1,\dots,a_s$ systems of generators of $I$ and $\fa$, respectively and let $\Phi=[c_{ij}]$ be an $r\times s$ matrix as above. For any integer $k$, the complex $_k\mathcal{Z}^+_\bullet =$$ _k\mathcal{Z}^+_\bullet(\ff, \mathbf{a},\Phi)$ constructed as above is called the \textbf{$k$-th residual approximation complex} with respect to the system of generators $\ff$ and $\mathbf{a}$ of $I$ and $\fa$, respectively and the matrix $\Phi$.
\end{Definition}
\noindent Among the whole family, we focus on  $k=0$ and we denote this complex simply as $\mathcal{Z}^+_\bullet$.
\[
\begin{tikzcd}\label{z0}
\mathcal{Z}^+_\bullet :& 0\arrow{r}& H^r_\mathfrak t(\mathcal D_{r+s})_{(0)}\arrow{r}&\cdots\arrow{r}&H^r_\mathfrak t (\mathcal{D}_{r})_{(0)} \arrow{r}{\tau_0}&R\arrow{r}&0.\\
\mathcal{Z}^+_\bullet: &0\arrow{r}& \mathcal{Z}^+_{s+1}\arrow{r} & \cdots\arrow{r}&\mathcal{Z}^+_{1} \arrow{r}{\tau_0}&\mathcal{Z}^+_{0} \arrow{r}&0.
\end{tikzcd},
\]
\noindent where $$\mathcal{Z}^+_{i}:=H^r_\mathfrak t (\mathcal{D}_{r+i-1})_{(0)} =\bigoplus_{j\geq i} Z_{r-s+j-1}^{\oplus n_{j}}$$ for all $i\geq 1$ and for some integer $n_{j}$. In particular $\mathcal{Z}^+_{s+1}=0$ if $\grade(I)>0$. Notice that the image of the map $\tau_0$ is an ideal of $R$. This ideal is called the \textit{disguised residual intersections} and it is denoted by $K(\mathbf{a},\ff,\Phi)$. The term disguised residual intersection finds its justification in its close relation with the concept of residual intersection, as elucidated in theorems \cite[Theorem 4.4]{HamidNaeliton} and \cite[Theorem 2.11]{hassanzadeh2012cohen}.

H. Hassanzadeh and V. Bouça \cite[Theorem 4.9, Propositions 4.11, 4.12, 4.15]{boucca2019residual} showed the desguised residual intersections are independent from the choice of the representation matrix and the systems of generators of $I$ and $\fa$. Moreover, they provided an alternative characterization of this ideal as follows

\begin{Theorem} Let $R$ be a  commutative ring and $\fa \subseteq I$ two finitely generated  ideals of $R$. Consider $\mathbf{f}=\:f_1,\dots,f_r$ and $\mathbf{a}=\:a_1,\dots,a_s$ system of generators of $I$ and $\fa$, respectively. Let $\Phi=[c_{ij}]$ be an $r\times s$ matrix in $R$ such that $[\mathbf{a}]=[\mathbf{f}]\cdot \Phi$.
Let  $K_\bullet(\ff;R)= R\langle e_1,\dots,e_r\:;\:\partial(e_i)=f_i\rangle$ be the Koszul complex equipped with the structure of differential graded algebra. Let 
$\zeta_j = \sum_{i=1}^r c_{ij}e_i$ for $1\leq j\leq s$, $\Gamma_\bullet=R\langle\zeta_1,\dots ,\zeta_s\rangle$ the algebra generated by the $\zeta$'s, and $Z_\bullet = Z_{\bullet}(\ff;R)$ be the algebra of Koszul cycles. Looking at the elements of degree $r$ in the sub-algebra of $K_{\bullet}$ generated by the product of $\Gamma_\bullet$ and $ Z_\bullet$, then 
$$K(\mathbf{a},\ff,\Phi)= \langle\Gamma_\bullet \cdot Z_\bullet\rangle _r =: \Kitt(\fa,I).$$
\end{Theorem}
\noindent We call $\Kitt(\fa,I)$ by the \textit{Kitt ideal} of $I$ with respect to $\fa$. 

In what follows, we will list some of the properties of $Kitt$ ideals. Some of these properties are straightforward to verify, while for others we cite appropriate references.
\begin{Theorem}\label{Main} Let $R$ be a Noetherian ring, $\fa\subseteq I$ two  ideals of $R$, $J:=\fa:_RI$ and $s$ a non-negative integer. Suppose that $I$ and $\fa$ are generated by $r$ and $s$ elements, respectively. 
\begin{enumerate}[(i)]
\item The ideal $\Kitt(\mathfrak a,I)$ does not depend on the choice of generators of $\fa$ and $I$ or the representative matrix {\normalfont \cite[Propositions 4.11, 4.12 and 4.15]{boucca2019residual}};
\item Denoting by $\tilde{H}_\bullet$ the sub-algebra of $K_{\bullet}$ generated by the representatives of Koszul homologies, then $\Kitt(\mathfrak a,I)= \fa+\langle\Gamma_\bullet \cdot \tilde{H}_\bullet\rangle_r$. Furthermore $\Kitt(\mathfrak a,I)\supseteq \Fitt_0(I/\fa)$ {\normalfont \cite[Theorem 4.23]{boucca2019residual}};  
\item One always has that $\Kitt(\fa,I)\subseteq J$ and they determine the same Zariski closed set. In addition if $\mu(I/\fa)\leq 1$, then $\Kitt(\mathfrak a,I)=J$ {\normalfont \cite[Theorem 4.4]{HamidNaeliton}, \cite[Theorem 2.11]{hassanzadeh2012cohen}};
\item  $\Kitt(0,I)= 0:_{R}I $, $\Kitt(\fa,\fa)=(1)$ and $\Kitt(\fa,(1))=\fa$; 
\item Given ideals $\fa_1 \subseteq \fa$ and $\fa \subseteq I_1 \subseteq I$, then $\Kitt(\fa,I)\subseteq \Kitt(\fa,I_{1})$ and $\Kitt(\fa_{1},I)\subseteq \Kitt(\fa,I)$ {\normalfont \cite[Theorem 4.23]{boucca2019residual}};
\item The Kitt ideal localizes. That is , $\Kitt(\fa,I)_{\fp}=\Kitt(\fa_{\fp},I_{\fp})$ for any prime ideal $\fp$ of $R$;
\item Kitt specializes. Specifically, for any regular element $f_{0}\in \fa$\begin{gather*}
    \frac{\Kitt(\fa,I)}{(f_{0})}=\Kitt\bigg(\frac{\fa}{(f_{0})},\frac{I}{(f_{0})}\bigg)
\end{gather*}  
{\normalfont \cite[Theorem 4.27]{boucca2019residual}};
\item Let $\grade(I)=g$ and  $J$  an $s$-residual intersection. If $s \leq g+1$, then $\Kitt(\fa,I)=J$ {\normalfont \cite[Proposition 5.10]{boucca2019residual}};
\item If $r\leq s$, then $$\Kitt(\fa,I)=\fa+\Fitt_{0}(I/\fa)+\sum_{{\{i_{1},\ldots,i_{r-2}\}\subset \{1,\ldots,s\}}}\Kitt((a_{i_{1}},\dots,a_{i_{r-2}}),I)$$ {\normalfont \cite[Proof of Theorem 4.7]{tarasova2021generators};}
\item If $r\geq s$, then $\Kitt(\fa,I)=\sum_{i=1}^{s}\Kitt((a_{1},\ldots,\hat{a_{i}},\ldots,a_{s}),I)+\zeta_{1}\cdot\ldots\cdot\zeta_{s}\cdot Z_{r-s}$;
\item Let $\fb$ be any ideal of $R$ with $I\cap \fb=0$. Then $\overline{\Kitt_R(\fa,I)}=\Kitt_{\bar{R}}(\bar{\fa},\bar{I})$, where $\bar{}$ is the image of canonical epimorphism of $R$ to $R/\fb$ { \normalfont \cite[Proof of Theorem 4.11]{tarasova2021generators}};
\item  Assume $\grade(I)\geq 1$ and set  $\Ht(J)=\alpha$. Suppose that $s\geq r$ and  $\mu(I_{\fp})=r$ for all prime ideals $\fp\supseteq I$ with $\Ht(\fp)\geq \alpha-1$. Then the arithmetic rank of $J$ is at most $s+rs-r^2+1$, \cite[Theorem 2]{BrunsSchwanzl} and \cite[Corollary4.3]{FreeApproch}. 
\item Suppose that $R$ is Cohen-Macaulay and $J$ is an $s$-residual intersection of $I$. If  $I$ satisfies $\SD_1$, then $\Kitt(\fa,I)=J$ and it is a Cohen-Macaulay ideal {\normalfont \cite[Theorem 5.1]{boucca2019residual}};
\item Suppose that $R$ is Cohen-Macaulay and $J$ is an $s$-residual intersection of $I$. If $I$ satisfies the $G_{s}$ condition and is weakly $(s-2)-$residually $S_{2}$, then $\Kitt(\mathfrak a,I)=J$. {\normalfont \cite[Theorem 4.11]{tarasova2021generators}.}
\end{enumerate}
\end{Theorem}
 There are a few definitions in the last theorem that need to be remembered.
\begin{Definition}\label{Gs} Let $(R,\fm)$ be a $d$-dimensional Cohen-Macaulay local ring, $I$ an ideal of $R$ with $\grade(I) = g$ and $s$ a integer
\begin{enumerate}[(i)]
\item One says that $I$ satisfies the $G_s$ condition if $\mu_{R_{\fp}}(I_{\fp})\leq \htt(\fp)$ for any prime ideal $ \fp \in \V(I)$ of height at most $s-1$;
\item  One says that $I$ is weakly $(s-2)-$residually $S_2$ if, for every $g\leq i\leq s-2$ and any geometric $(s-2)-$residual intersection $J=\fa:_RI$ of $I$, the ring $R/J$ satisfies the Serre's condition $S_{2}$, that is $\depth\big((R/J)_{\fp}\big) \geq \min\{2,\dim\big((R/J)_{\fp}\big)\};$
\item Let $\mathbf{f} = f_1,\dots,f_r$ be a system of generators of $I$ and $k$ an integer. One says that
$I$ satisfies the $\SD_k$ condition if $\depth(H_i(\mathbf{f};R))\geq \min\{d-g,d-r+i+k\}$ for all $i\geq 0$; also $\SD$ stands for
 $\SD_0$. Similarly, one says that $I$ satisfies the
sliding depth condition on cycles, $\SDC_k$, at level $t$, if
$\depth(Z_i(\mathbf{f};R))\geq \min\{d-r+i+k, d-g+2, d\}$ for all $r-g-t\leq i\leq r-g$. 
\item The arithmetic rank of an ideal $J$ is the minimum number of generators of an 
ideal with the same radical as that of $J$. In item $(xii)$ of \autoref{Main}, 
the quotient ideal $J = \mathfrak{a} : I$ is not necessarily a residual 
intersection. The condition on $I$ that $\mu(I_{\mathfrak{p}}) = r$ for all 
prime ideals $\mathfrak{p} \supseteq I$ with $\text{ht}(\mathfrak{p}) \geq \alpha - 1$ 
is called $r$-minimally generated in \cite{FreeApproch}. Classes of 
$r$-minimally generated ideals contain complete intersections; other examples 
can be found in \cite{FreeApproch}. In the case where $J = \mathfrak{a} : I$ 
is an $s$-residual intersection of a complete intersection, 
\cite[Theorem 4.7]{ARAresidualintersections} proves that the arithmetic rank of $J$ is at most 
$r + rs - r^2 + 1$, which improves the bound in item $(xii)$ of 
\autoref{Main} since $s \geq r$.
\end{enumerate}

\end{Definition}
In \cite{boucca2019residual}  it is conjectured that, for an $s-$residual intersection $J=\frak{a}:_RI$ in a Cohen-Macaulay ring, one always has $J=\Kitt(\frak{a},I)$.  However, in December 2022, L. Bus\'e and V. Bou\c{c}a  found a counterexample to that conjecture.
\begin{Example}
{\normalfont Let $\mathbb{K}$ be a field (infinite or of large characteristic), 
$R = \mathbb{K}[x_{1}, \ldots, x_{4}]$, and $I = (x_{3}, x_{4})(x_{1}, x_{2}^{2} - x_{3}x_{4})$. 
Let $M$ be a random matrix with entries of degrees $\{2, 2, 1, 1\}$, and 
define $\mathfrak{a}$ as the ideal generated by the entries of 
$\text{gens}(I) \cdot M$. Then $J = \mathfrak{a} :_R I$ is a $4$-residual 
intersection, but $J \neq \text{Kitt}(\mathfrak{a}, I)$.}
\end{Example}
In this example,   $\depth(Z_{2}(I))=2$ and $\Fitt_{2}(I)$ is $(x_{1},\dots,x_{4})-$primary ideal. Hence $I$ satisfies $G_{\infty}$ however it is not $\SD$ neither weakly $2$-residually $S_2$.

\section{Generic Kitt Ideals}
\subsection{Generic Residual Approximation Complex}
Let $R$ be a ring and $I$ a finitely generated ideal of $R$. In this subsection, we generalize the construction $k$-th residual approximation complex to the generic case. Let $s$ be a positive integer, $\mathbf{f}=\:f_1,\dots,f_r$ a system of generators of $I$ and 
$S = R[U_{ij}\:\:;\:\: 1 \leq i \leq r,\: 1 \leq j \leq s]$ the polynomial extension of $R$ in $rs$ indeterminates. Define $S^{\mathfrak{g}} := S[T_1,\dots,T_r]$ and $\gamma= \gamma_1,\dots, \gamma_s \in S^{\mathfrak{g}}$ such that
\begin{gather*}
    \begin{bmatrix}
\gamma_1 & \gamma_2 & \cdots  &\gamma_s 
\end{bmatrix} = \begin{bmatrix}
T_1 & T_2 & \cdots  & T_r 
\end{bmatrix}\begin{bmatrix}
U_{11} & U_{12}  &\cdots  & U_{1s} \\ 
U_{21} & U_{22} & \cdots & U_{2s} \\ 
\vdots & \vdots &\ddots  & \vdots \\ 
U_{r1} & U_{r2} &\cdots  & U_{rs} 
\end{bmatrix}. 
\end{gather*}
Consider the generic $\mathcal{Z}$-approximation complex $\mathcal{Z}^{\mathfrak{g}}_{\bullet}(\mathbf{f}) := \mathcal{Z}(\mathbf{f}, S)$
\[\mathcal{Z}^{\mathfrak{g}}_{\bullet}(\mathbf{f}):\:\:
\begin{tikzcd}
    0 \arrow{r} & Z_{r}(\mathbf{f}; S)\otimes_SS^{\mathfrak{g}}[-r] \arrow{r} & \cdots \arrow{r} & Z_{1}(\mathbf{f}; S)\otimes_SS^{\mathfrak{g}}[-1] \arrow{r} & S^{\mathfrak{g}} \arrow{r} & 0,
\end{tikzcd}
\] where $Z_{\bullet}(\mathbf{f}; S)$ is the algebra of cycles of the Koszul complex $K(\ff;S)$. Let $K_{\bullet}^{\mathfrak{g}}$ denote the Koszul complex $K(\gamma;S^{\mathfrak{g}})$ and define $\mathcal
D^{\mathfrak{g}}_{\bullet} = \Tot(K_{\bullet}^{\mathfrak{g}} \otimes_{S^{\mathfrak{g}}} \mathcal{Z}^{\mathfrak{g}}_{\bullet}(\mathbf{f}))$ the totalization of $K_{\bullet}^{\mathfrak{g}} \otimes_{S^{\mathfrak{g}}} \mathcal{Z}^{\mathfrak{g}}_{\bullet}(\mathbf{f})$. Thus
\[
    \mathcal{D}^{\mathfrak{g}}_{\bullet}:\:\: 
    \begin{tikzcd}
        0 \arrow{r}& \mathcal{D}^{\mathfrak{g}}_{r+s} \arrow{r} & \mathcal{D}^{\mathfrak{g}}_{r+s-1} \arrow{r} & \cdots \arrow{r} & \mathcal{D}^{\mathfrak{g}}_{1} \arrow{r} & \mathcal{D}^{\mathfrak{g}}_0 = S^{\mathfrak{g}} \arrow{r} & 0,
    \end{tikzcd}
\]
where 
\[ \mathcal{D}^{\mathfrak{g}}_{i} = \bigoplus_{\substack{l+j = i\\ 0\leq l \leq s,\: 0\leq j \leq r }}K_{l}^{\mathfrak{g}}\otimes_{S^{\mathfrak{g}}}(Z_{j}(\mathbf{f}, S)\otimes_SS^{\mathfrak{g}}[-j]) = \bigoplus_{\substack{l+j = i\\ 0\leq l \leq s,\: 0\leq j \leq r }}[Z_{j}(\mathbf{f}, S)\otimes_SS^{\mathfrak{g}}]^{\binom{s}{l}}[-i].
\]
Considering the ideal $\mathbf{t} = (T_1,\dots,T_r) \subseteq S^{\mathfrak{g}}$, one can apply the  same procedure of the construction of the Koszul-Čech complex to the double complex 
\begin{equation*}
    \mathcal{D}^{\mathfrak{g}}_{\bullet} \otimes {\mathcal{C}}^{\bullet}_{\mathbf{t}}(S^{\mathfrak{g}})
\end{equation*}
and gets a $\mathbb{Z}$-indexed family of complexes ${}_k\mathcal{Z}^{+\mathfrak{g}}_{\bullet}(s,\ff)$, where ${}_k\mathcal{Z}^{+\mathfrak{g}}_{\bullet}(s,\ff)$ is
\[
\begin{tikzcd}
    0 \arrow{r} & H^r_{\mathbf{t}}(\mathcal{D}^{\mathfrak{g}}_{r+s})_{(k)} \arrow{r} & \cdots \arrow{r}{\psi_k^\fg}  &  H^r_{\mathbf{t}}(\mathcal{D}^{\mathfrak{g}}_{k+r})_{(k)}\arrow{r}{\tau_k^\fg}  & (\mathcal{D}^{\mathfrak{g}}_{k})_{(k)} \arrow{r}{\mu_k^\fg} & \cdots    (\mathcal{D}^{\mathfrak{g}}_{0})_{(k)} \arrow{r} & 0
\end{tikzcd}
\]
for each $k \in \mathbb{Z}$. Analyzing each component of complex ${}_k\mathcal{Z}^{+\mathfrak{g}}_{\bullet}$, one concludes that the $i$th component is 
\begin{gather*}
    {}_k\mathcal{Z}^{+\mathfrak{g}}_{i} = 
    \bigoplus_{\substack{l+j = i\\ 0\leq l \leq s,\: 0\leq j \leq r }}(Z_{j}(\mathbf{f}; S)\otimes_SS^{\mathfrak{g}})^{\binom{s}{l}}_{(k-i)},\:\:\:\:\textrm{if $i \leq k$};\\
    {}_k\mathcal{Z}^{+\mathfrak{g}}_{i} = \bigoplus_{\substack{l+j =  r + i - 1\\ 0\leq l \leq s,\: 0\leq j \leq r }}(Z_{j}(\mathbf{f}; S)\otimes_SH^r_{\mathbf{t}}(S^{\mathfrak{g}}))^{\binom{s}{l}}_{(k-i)},\:\:\:\:\textrm{if $i > k$}.
\end{gather*}
Given $k$ an integer, the complex ${}_k\mathcal{Z}^{+\mathfrak{g}}_{\bullet}(s,\ff)$ as constructed above is called the {\it $s$-generic $k$-th residual approximation complex} with respect to $\mathbf{f}$. The image of the map $\tau_0^\fg$ is an ideal of $S$. This ideal is called \textit{$s$-generic disguised residual intersection}.
We denote this ideal by $K(s,\ff)$.
Similarly to the ordinary case, one has $K(s,\ff)=\Kitt((\alpha),(\ff)S)$, where 
\begin{gather*}
    \begin{bmatrix}
\alpha_1 & \alpha_2  & \cdots & \alpha_s
\end{bmatrix}
= 
\begin{bmatrix}
f_1 & f_2  & \cdots & f_r
\end{bmatrix}\begin{bmatrix}
U_{11} & U_{12} & \dots & U_{1s} \\ 
U_{21} & U_{22}  & \dots  & U_{2s} \\ 
\vdots & \vdots & \ddots & \vdots\\ 
U_{r1} & U_{r2} & \dots & U_{rs}
\end{bmatrix}.
\end{gather*}    

Throughout this work, we refer to the $s$-generic disguised residual intersection of $I$ with respect to the system of generators $\ff$ simply as the \textit{$s$-generic Kitt of $I$ with respect to $\ff$}, denoted by $\Kitt^{\fg}(s, \ff)$. Unlike the ordinary Kitt, the $s$-generic Kitt depends explicitly on the choice of the system of generators of $I$, as demonstrated in the following example.

\begin{Example} {\normalfont 
    Let $\mathbb{K}$ be a field, $R = \mathbb{K}[x,y]$ and $I$ an ideal of $R$ with two different system of generators $\ff=\:x^2 + y,\:x^5$ and $\ff'=\:x^2 + y,\:x^5 + x^2 + y$. Using Macaulay2 \footnote{ The calculation is inside the package {\sc Kitt} available in \url{https://sites.google.com/view/s-hamid-hassanzadeh/research?authuser=0}} one obtains
    \begin{gather*}
       \Kitt^{\mathfrak{g}}(2,\ff) = (x^5U_{22} + x^2U_{12} + yU_{12}, x^5U_{21} + x^2U_{11} + yU_{11}, U_{12}U_{21} - U_{11}U_{22}),\\
        \Kitt^{\mathfrak{g}}(2,\ff') = (x^5U_{22} + x^2U_{12} + x^2U_{22} + yU_{12}, x^5U_{21} + x^2U_{11} + x^2U_{21} + yU_{11} + yU_{21}, U_{12}U_{21} - U_{11}U_{22}).
    \end{gather*}
Hence $\Kitt^{\mathfrak{g}}(2,\ff) \neq \Kitt^{\mathfrak{g}}(2,\ff')$.} 
\end{Example} 
In this example, if we consider $U = [U_{ij}]$, $[\mathbf{a}] = [\mathbf{f}]U$, and $[\mathbf{a}'] = [\mathbf{f}']U$, we get that $\fa = (\mathbf{a})$ is different from $\fa' = (\mathbf{a}')$. Hence it is not surprising that we end up with $$\Kitt(\fa,IS)=\Kitt^{\mathfrak{g}}(2,\mathbf{f}) \neq \Kitt^{\mathfrak{g}}(2,\mathbf{f}') = \Kitt(\fa',IS).$$
Although the example above demonstrates the sensitivity of the s-generic Kitt to the choice of generators of $I$, it is important to note that this sensitivity is controlled. In other words,  the $s$-generic Kitt is unique up to \textit{universal equivalence}, which we recall the definition  in the following.
\begin{Definition}
Let $(R, I)$ and $(S,K)$ be two pairs, where $R$ and $S$ are rings and $I \subset R$, $K\subset S$ are ideals or $I = R$ or $K=S$. One says that $(R,I)$ and $(S,K)$ are universal equivalent if there are finite sets of indeterminates $X$, $Z$ over $R$, $S$, respectively and a ring isomorphism $\phi: R[X] \longrightarrow S[Z]$ such that $\phi(IR[X]) = KS[Z]$. 
\end{Definition}
It is straightforward to check that universal equivalence is indeed an 
equivalence relation. Drawing inspiration from Huneke and Ulrich's proof 
establishing the universal independence of the generic $s$-residual 
intersection with respect to the choice of the system of generators of $I$ 
presented in \cite{huneke1990generic}, we demonstrate that this universal 
equivalence similarly applies to the $s$-generic Kitt ideal. This proof 
relies on the fact that the Kitt structure commutes with ring automorphisms.
\begin{Proposition}
\label{EssentialEquivalence}
    Let $R$ be a ring, $I$ a finitely generated ideal of $R$ and $s$ a positive integer. Consider $\ff=\:f_1,\dots,f_r$ and $\ff'=\: f_1',\dots,f_{r'}'$ systems of generators of $I$ and let $R[U]$, $R[V]$ be polynomial extensions of $R$ in $rs$ and $r's$ indeterminates, respectively. Then $(R[U],\Kitt^{\mathfrak{g}}(s,\ff))$ and $(R[V],\Kitt^{\mathfrak{g}}(s,\ff'))$ are universally equivalent. 
\end{Proposition}
\noindent\textit{Proof:} Suppose firstly that $\ff'=\:f_1,\dots,f_r, g$. Let $a_1,\dots,a_s \in R[U]$ and $a_1',\dots,a_s' \in R[V]$ such that
\begin{gather*}
   \begin{bmatrix}
a_1 & a_2 & \cdots  &a_s 
\end{bmatrix} = \begin{bmatrix}
f_1 & f_2 & \cdots  & f_r 
\end{bmatrix}\begin{bmatrix}
U_{11} & U_{12}  &\cdots  & U_{1s} \\ 
U_{21} & U_{22} & \cdots & U_{2s} \\ 
\vdots & \vdots &\ddots  & \vdots \\ 
U_{r1} & U_{r2} &\cdots  & U_{rs} 
\end{bmatrix}\\
   \begin{bmatrix}
a_1' & a_2' & \cdots  &a_s' 
\end{bmatrix} = \begin{bmatrix}
f_1 & f_2 & \cdots  & f_r & g
\end{bmatrix}\begin{bmatrix}
V_{11} & V_{12}  &\cdots  & V_{1s} \\ 
V_{21} & V_{22} & \cdots & V_{2s} \\ 
\vdots & \vdots &\ddots  & \vdots \\ 
V_{r1} & V_{r2} &\cdots  & V_{rs}\\
V_{r+1,1} &V_{r+1,2} & \cdots & V_{r+1,s}
\end{bmatrix}
\end{gather*}
By definition of $s$-generic Kitt, one has that $$\Kitt^{\mathfrak{g}}(s,\ff) = \Kitt((a),(\ff)R[U])\:\:\:\:\:\:\:\:\:\:\:\:\textrm{and} \:\:\:\:\:\:\:\:\:\:\:\: \Kitt^{\mathfrak{g}}(s,\ff') = \Kitt((a'),(\ff,g)R[V]).$$ 
As $g$ belongs to $I$, there are $c_1,\dots,c_r \in R$ such that $g = \sum_{k=1}^rc_kf_k$. Thus consider the $R$-algebra homomorphism $\phi:R[U,V] \longrightarrow R[U,V]$ such that
\begin{gather*}
    \phi(U_{ij}) = V_{ij} + c_iV_{r+1,j}\:\:\:\textrm{for all $1\leq i \leq r$ and $1\leq j \leq s$};\\
    \phi(V_{ij}) = \begin{cases}
                    U_{ij},\:\:\textrm{if $1\leq i \leq r$ and $1\leq j \leq s$}\\
                    V_{r+1,j}, \:\:\textrm{if $i = r+1$.}
                    \end{cases}
\end{gather*}
Notice that $\phi$ is an isomorphism and 
$$\phi(a_i) =\phi\bigg(\sum_{k=1}^rf_kU_{ki}\bigg) = \sum_{k=1}^r(f_kV_{ki}) + V_{r+1,i}\bigg(\sum_{k=1}^rc_kf_k\bigg) = \sum_{k=1}^rf_kV_{ki} + gV_{r+1,i} = a_{i}',$$
which implies that $\phi((a)R[U,V]) = (a')R[U,V]$. Furthermore, one has $$\phi((\ff)R[U,V]) = (\ff)R[U,V] = (\ff,g)R[U,V],$$ because $(\ff)R[U,V] = (\ff,g)R[U,V]$. Then one has
\begin{gather*}
    \phi\big(\Kitt^{\mathfrak{g}}(s,\ff)(R[U])[V]\big) = \phi\big(\Kitt((a),(\ff)R[U])(R[U])[V]\big) \\ = \phi\big(\Kitt((a)R[U,V],(\ff)R[U,V])\big) = \Kitt(\phi((a)R[U,V]),\phi((\ff)R[U,V])) \\ = \Kitt((a')R[U,V],(\ff,g)R[U,V]) = \Kitt((a'),(\ff,g)R[V])(R[V])[U] \\ = \Kitt^{\mathfrak{g}}(s,\ff')((R[V])[U]). 
\end{gather*}
Here the equality in the second line comes from the behavior of Kitt under the action of homomorphism as studied in \cite{KevinThesis}.
Hence $(R[U],\Kitt^\mathfrak{g}(s,\ff))$ and $(R[V],\Kitt^\mathfrak{g}(s,\ff'))$ are universally equivalent. More generally, let $\ff$ and $\ff'$ be two systems of generators of $I$. Set $\ff,\ff' := f_1,\dots,f_r,f_1',\dots,f_{r'}'$. Using the previous argument, induction, and the fact that universal equivalence is an equivalence relation, one concludes that  
\begin{gather*}
    (R[U],\Kitt^\mathfrak{g}(s,\ff)) \equiv (R[W],\Kitt^\mathfrak{g}(s,\ff,\ff')),\\  (R[V],\Kitt^\mathfrak{g}(s,\ff')) \equiv (R[W],\Kitt^\mathfrak{g}(s,\ff,\ff')).
\end{gather*}
Hence the statement follows.\qed\\[0.3cm]

In his Ph.D. thesis, the second author delves into additional standard properties of generic Kitt ideals and explores their behavior under homomorphisms, including an examination of higher-order generic Kitt ideals.

\noindent A noteworthy property of the generic Kitt is its ability to specialize into the ordinary Kitt, as shown in the next proposition.
\begin{Proposition}
\label{Prop3.1}
    Let $R$ be a ring and $\fa\subseteq I$ finitely generated ideals of $R$. Consider $\mathbf{f} = f_1,\dots,f_r$ and  $\mathbf{a} = a_1,\dots,a_s$ systems of generators of $I$ and $\fa$, respectively and $\Phi = [c_{ij}]_{r\times s}$ such that
 \[
    \begin{bmatrix}
a_1 & a_2 & \cdots  &a_s 
\end{bmatrix} = \begin{bmatrix}
f_1 & f_2 & \cdots  & f_r 
\end{bmatrix}\begin{bmatrix}
c_{11} & c_{12}  &\cdots  & c_{1s} \\ 
c_{21} & c_{22} & \cdots &c_{2s} \\ 
\vdots & \vdots &\ddots  & \vdots \\ 
c_{r1} & c_{r2} &\cdots  &c_{rs} 
\end{bmatrix} = \begin{bmatrix}
f_1 & f_2 & \cdots  & f_r 
\end{bmatrix}\Phi. 
\]
Let $S = R[U_{ij}\:\:;\:\:1\leq i \leq r,\:\:1\leq j \leq s]$,  $\Kitt^{\mathfrak{g}}(s,\ff)$ be the $s$-generic Kitt ideal with respect the generating set $\mathbf{f}$ and the sequence $\mathbf{x} = U_{11} - c_{11},\dots, U_{rs} - c_{rs}$ in $S$, then  
$$\frac{\Kitt^{\mathfrak{g}}(s,\ff) + (\mathbf{x})}{(\mathbf{x})} = \Kitt(\mathfrak{a},I).$$
\end{Proposition}\vspace{0.3cm}
\noindent\textit{Proof:} Define the $R$-algebra homomorphism $\sigma: S \longrightarrow R$ with $\sigma(U_{ij}) = c_{ij}$. Notice that $(\mathbf{x}) \subseteq \ker(\sigma)$. Thus $\sigma$ factors in $\phi:S/(\mathbf{x}) \longrightarrow R$ such that the diagram below
$$
\begin{tikzcd}
    S   \arrow{d}{\pi} \arrow{r}{\sigma}                       & R\\
    S/(\mathbf{x})  \arrow{ur}{\phi} & 
\end{tikzcd}
$$
is commutative. Recall from definition of Kitt that $\Kitt(\fa,I) = \langle Z_{\bullet}(\ff;R) \cdot \Gamma \rangle_r$, where $\Gamma = \langle \zeta_1,\dots,\zeta_s\rangle \subseteq K_1(\ff;R)$ and $\zeta_j = \sum_{k=1}^rc_{kj}e_j$. 

Let $z = \rho e_1\wedge \cdots \wedge e_r \in \Kitt(\mathfrak{a},I)$ with $\rho\in R$. Then there are $n \in \mathbb{N}$, $L_{1_j} = \{i_{1_j},\dots,i_{k_j}\} \subseteq \{1,\dots,s\}$ with $|L_{1_j}| = k_j$ and $L_{2_j}$ such that $|L_2| = r-k_j$ such that
\begin{gather*}
    \rho e_1\wedge \cdots \wedge e_r =\sum_{j=1}^n \zeta_{i_{1_j}}\wedge\cdots \wedge\zeta_{i_{k_j}}\wedge z_{L_2,j},
\end{gather*}
where $z_{L_2,j} \in Z_{r-k_j}(\ff;R)$ for all $1\leq j \leq n$. Thus, denoting $\zeta_{i}^{\mathfrak{g}} = \sum_{k=1}^rf_kU_{ki}$ for $1\leq i\leq s$ and taking 
\begin{gather*}
    Z = \sum_{j=1}^n \zeta_{i_{1_j}}^{\mathfrak{g}}\wedge\cdots \wedge\zeta_{i_{k_j}}^{\mathfrak{g}}\wedge z_{L_2,j} \in \Kitt^{\mathfrak{g}}(s,\ff),
\end{gather*}
we get $\phi(\overline{Z}) = z$. Thus 
\begin{gather*}
    \Kitt(\mathfrak{a},I) \subseteq \phi\bigg(\frac{\Kitt^{\mathfrak{g}}(s,\ff) + (\mathbf{x})}{\mathbf{x}}\bigg).
\end{gather*}
Conversely, given $Z \in \Kitt^{\mathfrak{g}}(s,\ff)$, there are $n \in \mathbb{N}$, $L_{1_j} = \{i_{1_j},\dots,i_{k_j}\} \subseteq \{1,\dots,s\}$ with $|L_{1_j}| = k_j$ and $L_{2_j}$ such that $|L_2| = r-k_j$ such that
\begin{gather*}
    Z=\sum_{j=1}^n \zeta^{\mathfrak{g}}_{i_{1_j}}\wedge\cdots \wedge\zeta^{\mathfrak{g}}_{i_{k_j}}\wedge z_{L_2,j},
\end{gather*}
where  $z_{L_2,j} \in Z_{r-k_j}(\ff;S)$ for all $1 \leq j \leq n$. Now applying the map $\phi$ to $\overline{Z}$, one gets
\begin{gather*}
    \phi(z) = \sum_{j=1}^n \zeta_{i_{1_j}}\wedge\cdots \wedge\zeta_{i_{k_j}}\wedge z_{L_2,j}  \in \Kitt(\mathfrak{a},I)
\end{gather*}
and the statement follows.
\qed
\begin{Remark}\label{remarkiso} {\normalfont
    Observe that $\phi$ in \autoref{Prop3.1} is a ring isomorphism. Indeed, it is clear that $\phi$ is surjective. Consider $\overline{f(U)} \in \Ker(\phi)$. Using the division algorithm, one concludes that there is $g(U_{11},\dots,U_{rs}) \in (\mathbf{x})$ such that \begin{gather*}
    f(U_{11},\dots,U_{rs}) = g(U_{11},\dots,U_{rs}) + f(c_{11},\dots,c_{rs}).
\end{gather*} 
As $\overline{f(U)} \in \Ker(\phi)$, then $f(c_{11},\dots,c_{rs}) = \phi(\overline{f(U)}) = 0$, which implies that $f(U_{11},\dots,U_{rs}) = g(U_{11},\dots,U_{rs}) \in (\mathbf{x})$ and so $\overline{f(U)} = 0$. Hence $\phi$ is an isomorphism.}
\end{Remark}


\subsection{Deformation of Kitt}
In this subsection, we investigate conditions under which the ordinary Kitt deforms to the generic one. To begin, we recall the definition of deformation.
\begin{Definition}
    Let $(R,I)$ and $(S,J)$ be pairs, where $R$ and $S$ are rings, and $I\subset R$, $J \subset S$ are ideals with possibility of  $I = R$ or $J = S$. One says that $(S,J)$ is a deformation of $(R,I)$ if there exists a sequence $\mathbf{x} \subset S$, which is regular over $S$ and $S/J$ such that there exists a ring isomorphism $\phi:\:S/(\mathbf{x}) \longrightarrow R$, with $\phi\big((J + (\mathbf{x}))/(\mathbf{x})\big) = I$. 
\end{Definition}
When the rings $R$ and $S$ are understood, we simply say that the ideal $I$ deforms to $J$. With \autoref{Prop3.1} and its remark in mind, $(R,\Kitt(\fa,I))$ deforms to $(S,\Kitt^{\fg}(s,\ff))$ if and only if the sequence $\mathbf{x}$ is $S/\Kitt^{\fg}(s,\ff)$-regular. In the following theorem, we derive an exact sequence that will enable us to find conditions ensuring this property.  
\begin{Theorem}
\label{Prop1}
Considering the notation in the previous subsection, there exists the following exact sequence\\
\[\xymatrix{%
H_2(\mathcal{Z}^+_{\bullet}(\fa,I))  \ar@{->}[r] &      H_{2}(\mathbf{x};S/\Kitt^{\mathfrak{g}}(s,\ff)) \ar@{->}[r]^{d^{0,2}_2} &     H_{1}(\mathcal{Z}^{+\mathfrak{g}}_{\bullet}(s,\ff))\otimes_S S/(\mathbf{x}) \ar@{->}[ld] &   \\
& H_1(\mathcal{Z}^{+}_{\bullet}(\fa,I)) \ar@{->}[r] & H_1(\mathbf{x};S/\Kitt^{\mathfrak{g}}(s,\ff)) \ar@{->}[r] & 0.
}\]
\end{Theorem}
\begin{proof}
Let $S^{\mathfrak{g}} = S[T_1,\dots,T_r]$ be the polynomial extension of $S$ in $r$ indeterminates and $\gamma =\gamma_1,\dots,\gamma_s \in S^{\mathfrak{g}}$  such that
\begin{gather*}
    \begin{bmatrix}
\gamma_1 & \gamma_2 & \cdots  &\gamma_s 
\end{bmatrix} = \begin{bmatrix}
T_1 & T_2 & \cdots  & T_r 
\end{bmatrix}\begin{bmatrix}
U_{11} & U_{12}  &\cdots  & U_{1s} \\ 
U_{21} & U_{22} & \cdots & U_{2s} \\ 
\vdots & \vdots &\ddots  & \vdots \\ 
U_{r1} & U_{r2} &\cdots  & U_{rs} 
\end{bmatrix} = \begin{bmatrix}
T_1 & T_2 & \cdots  & T_r 
\end{bmatrix}\Phi^{\mathfrak{g}} 
\end{gather*}
Consider the complex $\mathcal{Z}^{+\mathfrak{g}}_{\bullet} := \mathcal{Z}^{+\mathfrak{g}}_{\bullet}(s,\ff)$ and $K_\bullet = K_{\bullet}(\mathbf{x};S)$ the Koszul complex of the sequence $\mathbf{x}$ in $S$. Consider $E^{\bullet,\bullet} = \mathcal{Z}^{+\mathfrak{g}}_{\bullet} \otimes_S K_{\bullet}$ the double complex obtained by tensor product of $\mathcal{Z}^{+\mathfrak{g}}_{\bullet}$ and $K_{\bullet}$. One displays this double complex in the second quadrant double complex as follows.
\[
\begin{tikzcd}
& 0 \arrow{d} & & 0 \arrow{d} & 0 \arrow{d} &\\
0 \arrow{r} & \mathcal{Z}^{+\mathfrak{g}}_{s}\otimes_S K_{rs} \arrow{r} \arrow{d} & \cdots \arrow{r} & \mathcal{Z}^{+\mathfrak{g}}_{1}\otimes_S K_{rs} \arrow{d} \arrow{r} & \mathcal{Z}^{+\mathfrak{g}}_{0}\otimes_S K_{rs} \arrow{r} \arrow{d} & 0 \\
& \vdots \arrow{d} & & \vdots \arrow{d} &\vdots \arrow{d} & \\
0 \arrow{r} & \mathcal{Z}^{+\mathfrak{g}}_{s}\otimes_S K_1 \arrow{r} \arrow{d} & \cdots \arrow{r} & \mathcal{Z}^{+\mathfrak{g}}_{1}\otimes_S K_1 \arrow{d} \arrow{r} & \mathcal{Z}^{+\mathfrak{g}}_{0}\otimes_S K_1 \arrow{r} \arrow{d} & 0 \\
  0 \arrow{r} & \mathcal{Z}^{+\mathfrak{g}}_{s}\otimes_S K_0 \arrow{r} \arrow{d} & \cdots \arrow{r} & \mathcal{Z}^{+\mathfrak{g}}_{1}\otimes_S K_0 \arrow{d} \arrow{r} & \mathcal{Z}^{+\mathfrak{g}}_{0}\otimes_S K_0 \arrow{r} \arrow{d} & 0 \\
  & 0 & & 0 & 0 & 
\end{tikzcd}
\]
Observe that $(-i)$-th column of this double complex is the Koszul complex of sequence $\mathbf{x}$ with coefficients in $\mathcal{Z}^{+\mathfrak{g}}_{i}$. Note that $\mathcal{Z}^{+\mathfrak{g}}_{i} = \mathcal{Z}^{+}_i(\mathbf{a},\ff)[U_{ij}]$. Hence $\mathbf{x}$ is regular on $\mathcal{Z}^{+\mathfrak{g}}_{i}$, which implies that
    $H_j(\mathbf{x};\mathcal{Z}^{+\mathfrak{g}}_{i}) = 0$
for all $1\leq j \leq rs$. Moreover $H_0(\mathbf{x};\mathcal{Z}^{+\mathfrak{g}}_{i}) = \mathcal{Z}^{+\mathfrak{g}}_{i}\otimes_S S/(\mathbf{x}) \cong \mathcal{Z}^{+}_{i}$ for all $0\leq i \leq s$. Thus the first page of vertical spectral ${}^1E^{\bullet,\bullet}_{\ver}$ gives ${}^1E^{-p,q}_{\ver}=0$ for all $p$ and all $q>0$ and
\[
\begin{tikzcd}
  {}^1E^{\bullet,0}_{\ver}=\quad 0 \arrow{r} & \mathcal{Z}^{+}_{s} \arrow{r}  & \cdots \arrow{r} & \mathcal{Z}^{+}_{1}  \arrow{r} & \mathcal{Z}^{+}_{0} \arrow{r}  & 0. 
\end{tikzcd}
\]
The second page of vertical spectral ${}^2E^{\bullet,\bullet}_{\ver}$ gives ${}^2E^{-p,q}_{\ver}=0$ for all $p$ and all $q>0$ and
\[
\begin{tikzcd}
{}^2E^{\bullet,0}_{\ver}=\quad
   H_s(\mathcal{Z}^{+}_{\bullet})   & \cdots  & H_{1}(\mathcal{Z}^{+}_{\bullet})  & H_0(\mathcal{Z}^{+}_{\bullet}).
\end{tikzcd}
\] 

Since the vertical spectral sequence collapses on page two, we have that ${}^2E^{\bullet,\bullet}_{\ver} = {}^\infty E^{\bullet,\bullet}_{\ver}$. Hence, by convergence theorem, we conclude that for all $i\geq 0$
\begin{gather*}
    H_{-i}(\Tot(E^{\bullet,\bullet})) = H_i(\mathcal{Z}^{+}_{\bullet}).
\end{gather*}
 Now we analyze the horizontal spectral sequence. Since $K_i$ is a flat 
$S$-module for all $0 \leq i \leq rs$, the first page of the horizontal 
spectral sequence ${}^1E^{\bullet,\bullet}_{\text{hor}}$ is given by
\[
\begin{tikzcd}
 0 \arrow{d} & & 0 \arrow{d} & 0 \arrow{d} \\
 H_{s}(\mathcal{Z}^{+\mathfrak{g}}_{\bullet})\otimes_S K_{rs}  \arrow{d} & \cdots  & H_{1}(\mathcal{Z}^{+\mathfrak{g}}_{\bullet})\otimes_S K_{rs} \arrow{d}  & H_{0}(\mathcal{Z}^{+\mathfrak{g}}_{\bullet})\otimes_S K_{rs} \arrow{d}  \\
 \vdots \arrow{d} & & \vdots \arrow{d} &\vdots \arrow{d}  \\
 H_{s}(\mathcal{Z}^{+\mathfrak{g}}_{\bullet})\otimes_S K_1 \arrow{d}   & \cdots  & H_{1}(\mathcal{Z}^{+\mathfrak{g}}_{\bullet})\otimes_S K_1 \arrow{d}  & H_{0}(\mathcal{Z}^{+\mathfrak{g}}_{\bullet})\otimes_S K_1  \arrow{d}  \\
 H_{s}(\mathcal{Z}^{+\mathfrak{g}}_{\bullet})\otimes_S K_0  \arrow{d} & \cdots & H_{1}(\mathcal{Z}^{+\mathfrak{g}}_{\bullet})\otimes_S K_0 \arrow{d} & H_{0}(\mathcal{Z}^{+\mathfrak{g}}_{\bullet})\otimes_S K_0  \arrow{d}  \\
 0 & & 0 & 0 
\end{tikzcd}
\]
Since the $(-i)$-column of ${}^1E^{\bullet,\bullet}_{\hor}$ is the Koszul complex of the sequence $\mathbf{x}$ with coefficients in $H_i(\mathcal{Z}^{+\mathfrak{g}}_{\bullet})$, the second page of horizontal spectral ${}^2E^{\bullet,\bullet}_{\hor}$ gives
\[
\begin{tikzcd}
 H_{rs}(\mathbf{x};H_{s}(\mathcal{Z}^{+\mathfrak{g}}_{\bullet}))   & \cdots  & H_{rs}(\mathbf{x};H_{1}(\mathcal{Z}^{+\mathfrak{g}}_{\bullet}))  & H_{rs}(\mathbf{x};H_{0}(\mathcal{Z}^{+\mathfrak{g}}_{\bullet}))   &  \\
 \vdots  & & \vdots  &\vdots  & \\
 H_{1}(\mathbf{x};H_{s}(\mathcal{Z}^{+\mathfrak{g}}_{\bullet}))   & \cdots & H_{1}(\mathbf{x};H_{1}(\mathcal{Z}^{+\mathfrak{g}}_{\bullet}))  &  H_{1}(\mathbf{x};H_{0}(\mathcal{Z}^{+\mathfrak{g}}_{\bullet}))  & \\
 H_{0}(\mathbf{x};H_{s}(\mathcal{Z}^{+\mathfrak{g}}_{\bullet}))   & \cdots  & H_{0}(\mathbf{x};H_{1}(\mathcal{Z}^{+\mathfrak{g}}_{\bullet}))  & H_{0}(\mathbf{x};H_{0}(\mathcal{Z}^{+\mathfrak{g}}_{\bullet}))  & 
\end{tikzcd}
\]
Note that \begin{gather*}
    H_{j}(\mathbf{x};H_{0}(\mathcal{Z}^{+\mathfrak{g}}_{\bullet})) = H_{j}(\mathbf{x};S/\Kitt^{\mathfrak{g}}(s,\ff))\:\:\:\textrm{for all $0 \leq j \leq rs$};\\
    H_{0}(\mathbf{x};H_{i}(\mathcal{Z}^{+\mathfrak{g}}_{\bullet})) = H_{i}(\mathcal{Z}^{+\mathfrak{g}}_{\bullet})\otimes_S S/(\mathbf{x}) \:\:\:\textrm{for all $0 \leq i \leq s$}.
\end{gather*} 
Moreover, observe that the differential $d_2$ has bidegree $(-1,-2)$. This means that $${}^2E^{0,0}_{\hor} = {}^\infty E^{0,0}_{\hor},\:\:\:\:\:\:\:\:\:\:\:\:\:\:\:\:\:\:{}^2E^{0,1}_{\hor} = {}^\infty E^{0,1}_{\hor},\:\:\:\:\:\:\:\:\:\:\:\: \:\:\:\:\:\: {}^\infty E^{-1,0}_{\hor} = \coker(d_2^{0,2}). $$  
Thus we can construct the following exact sequence
\[
\begin{tikzcd}
    0\arrow{r}&\ker(d^{0,2}_2) \arrow{r} &H_{2}(\mathbf{x};S/\Kitt^{\mathfrak{g}}(s,\ff)) \arrow{r}{d^{0,2}_2} & H_{1}(\mathcal{Z}^{+\mathfrak{g}}_{\bullet})\otimes_S S/(\mathbf{x}) \arrow{r} & \coker(d_2^{0,2}) \arrow{r} & 0.
\end{tikzcd}
\]
Observe that $\ker(d^{0,2}_2) \cong {}^3E^{0,2}_{\hor} = {}^\infty E^{0,2}_{\hor}$. By convergence theorem, there is a filtration of $H_2(\Tot(E^{\bullet,\bullet}))$
\begin{gather*}
    0 \subseteq \mathcal{F}_0 \subseteq \mathcal{F}_1 \subseteq \mathcal{F}_2 = H_{-2}(\Tot(E^{\bullet,\bullet})) 
\end{gather*}
such that
\begin{gather*}
    \mathcal{F}_2/\mathcal{F}_1 = H_{-2}(\Tot(E^{\bullet,\bullet}))/\mathcal{F}_1 = {}^\infty E^{0,2}_{\hor} = \ker(d^{0,2}_2),
\end{gather*}
which implies that there is the following exact sequence
\[
\begin{tikzcd}
    0 \arrow{r} &\mathcal{F}_1\arrow{r}&H_{-2}(\Tot(E^{\bullet,\bullet})) \arrow{r}& \ker(d^{0,2}_2) \arrow{r} & 0.
\end{tikzcd}
\]
Since $H_{-2}(\Tot(E^{\bullet,\bullet})) \cong H_2(\mathcal{Z}^+_{\bullet})$, splicing the two sequences above, we obtain
\[
\begin{tikzcd}
    H_2(\mathcal{Z}^+_{\bullet}) \arrow{r}& H_{2}(\mathbf{x};S/\Kitt^{\mathfrak{g}}(s,\ff)) \arrow{r}{d^{0,2}_2} & H_{1}(\mathcal{Z}^{+\mathfrak{g}}_{\bullet})\otimes_S S/(\mathbf{x}) \arrow{r} & \coker(d_2^{0,2}) \arrow{r} & 0.
\end{tikzcd}
\]
As ${}^\infty E^{-1,0}_{\hor} = \coker(d_2^{0,2})$ and ${}^\infty E^{0,1}_{\hor} = H_1(\mathbf{x};S/\Kitt^{\mathfrak{g}}(s,\ff))$, the convergence theorem tells us that there exists a filtration of $H_{-1}(\Tot(E^{\bullet,\bullet}))$
\begin{gather*}
    0 \subseteq \mathcal{F}_0 \subseteq \mathcal{F}_1 = H_{-1}(\Tot(E^{\bullet,\bullet})) 
\end{gather*}
such that
\begin{gather*}
    \mathcal{F}_0 =  {}^\infty E^{-1,0}_{\hor} =\coker(d_2^{0,2})\:\:\:\:\:\:\:\:\textrm{and} \:\:\:\:\:\:\:\: \frac{H_{-1}(\Tot(E^{\bullet,\bullet}))}{\coker(d_2^{0,2})} =  {}^\infty E^{0,1}_{\hor} = H_1(\mathbf{x};S/\Kitt^{\mathfrak{g}}(s,\ff)). 
\end{gather*}
Thus we also can construct the following exact sequence
\[
\begin{tikzcd}
  0\arrow{r} & \coker(d^{0,2}_2) \arrow{r} &  H_{-1}(\Tot(E^{\bullet,\bullet})) \arrow{r} &  H_1(\mathbf{x};S/\Kitt^{\mathfrak{g}}(s,\ff))  \arrow{r} & 0.
\end{tikzcd}
\]
Finally since $H_{-1}(\Tot(E^{\bullet,\bullet})) = H_1(\mathcal{Z}^+_{\bullet})$, splicing the two sequences above, we get\\[0.3cm]
\xymatrix{%
H_2(\mathcal{Z}^+_{\bullet})  \ar@{->}[r] & H_{2}(\mathbf{x};S/\Kitt^{\mathfrak{g}}(s,\ff)) \ar@{->}[r]^{d^{0,2}_2} & H_{1}(\mathcal{Z}^{+\mathfrak{g}}_{\bullet})\otimes_S S/(\mathbf{x}) \ar@{->}[r] &  H_1(\mathcal{Z}^{+}_{\bullet}) \ar@{->}[dl] \\
&& H_1(\mathbf{x};S/\Kitt^{\mathfrak{g}}(s,\ff)) \ar@{->}[r] & 0.
}
\end{proof}

As a special case of the above theorem we have 
\begin{Corollary}
\label{cor1}
If $\mathcal{Z}^{+\mathfrak{g}}_{\bullet}(s,\ff)$ is an acyclic complex, then for all $0 \leq i \leq rs$,
 $$H_{i}(\mathcal{Z}^{+}_{\bullet}(\mathbf{a},\ff))  = H_i(\mathbf{x}; S/\Kitt^{\mathfrak{g}}(s,\ff)).$$ 
 \end{Corollary}
\noindent\textit{Proof:} Indeed if $\mathcal{Z}^{+\mathfrak{g}}_{\bullet}$ is acyclic, then the second page of horizontal spectral of double complex $E^{\bullet,\bullet} = \mathcal{Z}^{+\mathfrak{g}}_{\bullet} \otimes_S K_{\bullet}$ collapses becoming
\[
\begin{tikzcd}
0   & \cdots  & 0  & H_{rs}(\mathbf{x};H_{0}(\mathcal{Z}^{+\mathfrak{g}}_{\bullet}))  \\ 
 \vdots  & & \vdots  &\vdots  \\
 0   & \cdots  & 0  & H_{0}(\mathbf{x};H_{0}(\mathcal{Z}^{+\mathfrak{g}}_{\bullet}))   \\
\end{tikzcd}
\]
Thus, by the convergence theorem, we get $H_{-i}(\Tot(E^{\bullet,\bullet})) \cong H_{i}(\mathbf{x};H_{0}(\mathcal{Z}^{+\mathfrak{g}}_{\bullet}))$ for all $0 \leq i \leq rs$. Since $H_{0}(\mathcal{Z}^{+\mathfrak{g}}_{\bullet}) = S/\Kitt^{\mathfrak{g}}(s,\ff)$, comparing with the vertical spectral sequence, we conclude 
\begin{gather*}
    H_{i}(\mathbf{x};S/\Kitt^{\mathfrak{g}}(s,\ff)) \cong H_i(\mathcal{Z}^{+}_{\bullet})
\end{gather*}
for all $0 \leq i \leq rs$
\qed\\[0.3cm]
\noindent In particular, if $R$ is a Noetherian ring and $\mathcal{Z}^{+\mathfrak{g}}_{\bullet}(s,\ff)$ is an acyclic complex, \autoref{cor1} says that the complex $\mathcal{Z}^{+}_{\bullet}$ is rigid. 
\begin{Corollary}
\label{cor4.3}
    If $H_1(\mathcal{Z}^{+\mathfrak{g}}_{\bullet}(s,\ff)) =0$, then $H_1(\mathcal{Z}^{+}_{\bullet}(\mathbf{a},\ff)) \cong H_{1}(\mathbf{x};S/\Kitt^{\mathfrak{g}}(s,\ff)).$ Moreover, if $R$ is a Noetherian ring, the sequence $\mathbf{x}=U_{11} - c_{11},\dots, U_{rs} - c_{rs}$ is regular over $S/\Kitt^{\mathfrak{g}}(s,\ff)$ if and only if $H_1(\mathcal{Z}^+_{\bullet}(\mathbf{a},\ff)) = 0$.
\end{Corollary}
\begin{Proposition}
\label{prop2}
Suppose that $R$ is a Noetherian ring. If $\mathcal{Z}^{+\mathfrak{g}}_{\bullet}(s,\ff)$ is acyclic and $\mathbf{x}=\: U_{11} - c_{11},\dots, U_{rs} - c_{rs}$ is regular on $S/\Kitt^{\mathfrak{g}}(s,\ff)$, then $\mathcal{Z}^+_{\bullet}(\mathbf{a},\ff)$ is acyclic. The converse holds if $(R,\mathfrak{m})$ is Noetherian local and $\mathfrak{a} \subseteq \mathfrak{m}I$.
\end{Proposition}
\noindent\textit{Proof:} Suppose that $Z^{+\mathfrak{g}}_{\bullet}$ is acyclic and $\mathbf{x}$ is regular on $S/\Kitt^{\mathfrak{g}}(s,\ff)$. By \autoref{cor1}, we have have $H_{i}(\mathcal{Z}^{+}_{\bullet}) \cong H_{i}(\mathbf{x};S/\Kitt^{\mathfrak{g}}(s, \ff)) = 0$ for all $i\geq 1$, which implies that $\mathcal{Z}^{+}_{\bullet}$ is an acyclic complex. 

Conversely suppose the complex $\mathcal{Z}^+_{\bullet}$ is acyclic. Firstly note that ${}^2E^{p,q}_{\ver} = 0$ except if $(p,q) = (0,0)$, which implies that $H_{-i}(\Tot(E^{\bullet,\bullet})) = 0$ for all $i\geq 1$. Analyzing the horizontal spectral on infinite, we obtain
\[
\begin{tikzcd}
 \vdots & \vdots\\ 
* &   {}^2E^{0,1}_{\hor}  \\
{}^\infty E^{-1,0}_{\hor}  & H_0(\mathbf{x};H_0(\mathcal{Z}^{+\mathfrak{g}}_{\bullet}))
\end{tikzcd}
\] By convergence theorem, one knows that there exists a filtration of $H_{-1}(\Tot(E^{\bullet,\bullet}))$
\begin{gather*}
    0 = \mathcal{F}_0 \subseteq \mathcal{F}_1 = H_{-1}(\Tot(E^{\bullet,\bullet})) = 0
\end{gather*}
such that ${}^2E^{0,1}_{\hor} = \mathcal{F}_1/\mathcal{F}_0$. However, since $\mathcal{F}_1 = \mathcal{F}_2 = 0$, one has that ${}^2E^{0,1}_{\hor} = 0$. As $$H_1(\mathbf{x};S/\Kitt^{\fg}(s, \ff)) = {}^2E^{0,1}_{\hor}  = 0$$ and $R$ is a Noetherian ring, the Koszul rigidity tells us that $\mathbf{x}$ is a regular sequence on $S/\Kitt^{\mathfrak{g}}(s,\ff)$.

\noindent Now we will prove by induction on $i$ that $H_i(\mathcal{Z}^{+\mathfrak{g}}_{\bullet}) = 0$ for all $i\geq 1$. For $i=1$, observe that 
\begin{gather*}
    {}^\infty E^{-1,0}_{\hor} = \coker(d_2^{0,2}) = {}^2 E^{-1,0}_{\hor}
\end{gather*}
because ${}^2 E^{0,2}_{\hor} = 0$. Thus there exists a filtration of $H_{-1}(\Tot(E^{\bullet,\bullet}))$
\begin{gather*}
    0 = \mathcal{F}_0 \subseteq \mathcal{F}_1 = H_{-1}(\Tot(E^{\bullet,\bullet})) = H_1(\mathcal{Z}^+_{\bullet}) = 0
\end{gather*}
such that $  H_0(\mathbf{x};H_1(\mathcal{Z}^{+\mathfrak{g}}_{\bullet}))= {}^2E^{-1,0}_{\hor} = \mathcal{F}_1 = 0$. Since
\begin{gather*}
    H_1(\mathcal{Z}^{+\mathfrak{g}}_{\bullet}) / (\mathbf{x})H_1(\mathcal{Z}^{+\mathfrak{g}}_{\bullet}) = H_0(\mathbf{x};H_1(\mathcal{Z}^{+\mathfrak{g}}_{\bullet})) = 0,
\end{gather*}
we have that $H_1(\mathcal{Z}^{+\mathfrak{g}}_{\bullet}) = (\mathbf{x})H_1(\mathcal{Z}^{+\mathfrak{g}}_{\bullet})$. Note $H_1(\mathcal{Z}^{+\mathfrak{g}}_{\bullet})$ is a finitely generated graded $S$-module, since $S$ is Noetherian graded ring. Moreover, since $\mathfrak{a} \subseteq \mathfrak{m}I$, the elements $c_{ij} \in \mathfrak{m}$, which implies that the ${}^*$maximal ideal $(\mathfrak{m},U_{ij})$ contains $(\mathbf{x})$. Finally, as
\begin{gather*}
    H_1(\mathcal{Z}^{+\mathfrak{g}}_{\bullet})_{(\mathfrak{m},U_{ij})} = (\mathbf{x})H_1(\mathcal{Z}^{+\mathfrak{g}}_{\bullet})_{(\mathfrak{m},U_{ij})},
\end{gather*}
the graded Nakayama lemma says that $H_1(\mathcal{Z}^{+\mathfrak{g}}_{\bullet})_{(\mathfrak{m},U_{ij})} = 0$, which means that $H_1(\mathcal{Z}^{+\mathfrak{g}}_{\bullet}) = 0$. 

Suppose that we have proved that $H_i(\mathcal{Z}^{+\mathfrak{g}}_{\bullet}) = 0$ for all $ 1 \leq i < n\leq s$. Thus the second page of horizontal spectral sequence ${}^2E^{\bullet,\bullet}_{\hor}$ becomes
\[
\begin{tikzcd}
*   & *  & H_{rs}(\mathbf{x};H_{n}(\mathcal{Z}^{+\mathfrak{g}}_{\bullet}))& 0 &  \cdots & 0 & 0  \\
\vdots   & \vdots  & \vdots & \vdots &  \ddots &\vdots &\vdots  \\
 *   & *  & H_{1}(\mathbf{x};H_{n}(\mathcal{Z}^{+\mathfrak{g}}_{\bullet}))& 0 &  \cdots & 0 & 0  \\
 *   & *  &  H_n(\mathcal{Z}^{+\mathfrak{g}}_{\bullet})/(\mathbf{x})H_n(\mathcal{Z}^{+\mathfrak{g}}_{\bullet})  &  0 & \cdots & 0&H_0(\mathbf{x};H_0(\mathcal{Z}^{+\mathfrak{g}}_{\bullet}))  \\
\end{tikzcd}
\]
By convergence theorem, we conclude that $$ H_n(\mathcal{Z}^{+\mathfrak{g}}_{\bullet})/(\mathbf{x})H_n(\mathcal{Z}^{+\mathfrak{g}}_{\bullet}) = H_{-n}(\Tot(E^{\bullet,\bullet})) = H_{n}(\mathcal{Z}^{+}_{\bullet}) = 0.$$
Finally, proceeding as the case $i=1$, we conclude that $H_n(\mathcal{Z}^{+\mathfrak{g}}_{\bullet}) = 0$.
\qed
\begin{Remark}
Note that the local restriction on $R$ mentioned in the \autoref{prop2} is not  required for establishing that the sequence $\mathbf{x}$ is regular on $S/\Kitt^{\mathfrak{g}}(s,\ff)$.
\end{Remark}
\begin{Corollary}\label{deformg+1}
    Let $R$ be a Cohen-Macaulay local ring and $J = \mathfrak{a}:_R I$ an $s$-residual intersection of $I$.  Under any of the following conditions
    \begin{enumerate}[(i)]
        \item $s \leq \htt(I) + 1$;
        \item $R$ Gorenstein, $s = \htt(I) + 2$ and $I^{\unm}$ Cohen-Macaulay ideal.
    \end{enumerate}
     The sequence $\mathbf{x}= U_{11} - c_{11},\dots, U_{rs} - c_{rs}$ is regular on $S/\Kitt^{\mathfrak{g}}(s,\ff)$. In particular, $\Kitt(\mathfrak{a},I)$ deforms to $\Kitt^{\mathfrak{g}}(s,\ff)$.
\end{Corollary}
\noindent\textit{Proof:} In fact, by \cite[Corollary 2.9 (b)]{hassanzadeh2012cohen}, the complex $\mathcal{Z}^+_{\bullet}$ is acyclic. By  \autoref{prop2} and its remark, we conclude that the sequence $\mathbf{x}$ is regular on $S/\Kitt^{\mathfrak{g}}_s(s,\ff)$. Since $\mathbf{x}$ also is an $S$-regular sequence, then $\Kitt(\mathfrak{a},I)$ deforms to $\Kitt^{\mathfrak{g}}(s,\ff)$
\qed\\[0.3cm]

\section{Generic Residual}
\indent Let $R$ be a Noetherian ring, $I$ a non-zero ideal of $R$ and $s\geq \max\{1,\htt(I)\}$ a positive integer. Consider $\ff= f_1,\dots,f_r$ a system of generators of $I$. Let $S = R[U_{ij}\:\:;\:\: 1 \leq i \leq r, 1 \leq j \leq s]$ be the polynomial extension of $R$ in $rs$ indeterminates and $\mathbf{a}'=\:a_1',\dots,a_s' \in S$ such that
$$\begin{bmatrix}
a_1' & a_2' & \cdots & a_s'
\end{bmatrix}=
\begin{bmatrix}
f_1 & f_2 & \cdots & f_r
\end{bmatrix}
\begin{bmatrix}
U_{11} & U_{12} & \cdots  & U_{1s} \\ 
U_{21} & U_{22} & \cdots & U_{2s} \\ 
\vdots & \vdots & \ddots & \vdots \\ 
U_{r1} & U_{r2} & \cdots  &  U_{rs}
\end{bmatrix}.$$
\begin{Definition}\label{R(s,f)}
    Given the notation established above, one defines the {\bf generic $s$-residual} of the system of generators $\ff$ the ideal $((\mathbf{a'}):_S IS)$ of $S$ and it is denoted as $\R(s,\ff) \subseteq S$.
\end{Definition}
Notice that this definition draws inspiration from the definition of,  $\RI(s,\ff) $, the {\bf generic $s$-residual intersection}  by  Huneke and Ulrich in \cite{huneke1988residual}. The difference lies in the omission of the $G_{s+1}$ condition. This property is important in ensuring that $\R(s,\ff)$ indeed is an $s$-residual intersection of $IS$ as one can verify in \cite[Lemma 2.2]{huneke1988residual}. 

One sees from the following proposition that having height at least $s$ is an extreme case for generic $s$-residuals.

We keep the notation of the beginning of this section.
\begin{Proposition}
    Let $(R,\fm)$ be a  local Noetherian   formally equidimensional ring, $I=(\ff)$ an ideal and $s\geq \max\{1,\htt(I)\}$ an integer.  If $I$ is not  a nilpotent ideal, then
    $\htt(\R(s,\ff))\leq s$.
\end{Proposition}
\begin{proof} Suppose by contradiction that $\htt(\R(s,\ff))> s$. 
Denote $\fa = (\mathbf{a}') \subseteq S$ and let $t$ be an indeterminate over $S$. Observe that 
$${    (\fa, t):_{S[t]}(I,t)= ((\fa:_SIS),t)}$$
Furthermore, note that $\R(s,\ff) = \fa:_SIS$ and
\begin{gather*}
      \htt\big( (\fa, t):_{S[t]}(IS,t)\big) = \htt((\fa:_SI),t) = \htt(\R(s,\ff))+1 > s+1
\end{gather*}
Since formal equidimensionality is stable under localization and polynomial extension \cite[Theorem 18.13]{HeremannBlowing}, $S[t]$ is a locally formally equidimensional ring. Then applying  the same idea of the  proof of  \cite[Proposition 3]{RemarksUlrich} for the non-local ring $S[t]$, one has that $(I,t)S \subseteq \overline{(\mathfrak{a},t)}$ where $ \overline{(\mathfrak{a},t)}$ is the integral closure of the ideal $(\mathfrak{a},t)$.  Recall that $\fa$ is a generic subideal of $IS$.  
By specialization $U_{ij}$ to $0$, one concludes that $(I,t)$ is integral over the principal ideal $(t)$. Since $t$ is an indeterminate, one has $I$ is a nilpotent ideal which is a contradiction. Hence $\htt(\R(s,\ff))\leq s$.
\end{proof}
Notice that the structure of generic $s$-residual depends on the choice of generators of $I$.
Similarly, to what occurs with the generic Kitt, it can be demonstrated that the generic $s$-residual of the same ideal with different generators is also unique up to universal equivalence. 

In his Ph.D. thesis, V. Bouça  proved that for any $s$-generated ideal $\fa \subseteq (\ff)$ in local Cohen-Macaulay rings, one has  $\htt(\Kitt^\fg(s,\ff)) \geq \htt(\Kitt(\fa,(\ff))).$ We noted this conclusion may not always hold. The  following example is a counterexample.
\begin{Example}
{\normalfont Let $\mathbb{K}$ be a field, $R = \mathbb{K}[x_0,\dots,x_3]$ and  $T = \mathbb{K}[s,t]$. Let $\fa=(x_2^3 - x_1x_3^2,\:x_0x_2^2 - x_1^2x_3,\: x_1^3 - x_0^2x_2) \subset R$, a sub-ideal of the defining ideal of the monomial curve $k[s^{4},s^{3}t,st^{3},t^{4}]$.
Let  \begin{gather*}
  I = \fa:_R(x_0,x_1, x_2, x_3) = (x_2^3-x_1x_3^2,\: x_0x_2^2-x_1^2x_3,\:x_1^3-x_0^2x_2,\: x_1^2x_2^2 -x_0x_1x_2x_3) \\ =: (f_1,f_2,f_3,f_4) \supset \fa.  
\end{gather*}
Since $J = \fa:_R I = (x_0,x_1,x_2,x_3)$, one concludes that $\htt(\Kitt(\fa,I)) = \htt(J) = 4$. Now consider $S = R[U_{1},\dots,U_{12}]$  the polynomial extension of $R$ in twelve indeterminates and let $\alpha= \alpha_1,\alpha_2,\alpha_3 \in S$ where
\[
\begin{bmatrix}
\alpha_1 & \alpha_2 & \alpha_3
\end{bmatrix} =
\begin{bmatrix}
f_1 & f_2 & f_3 & f_4
\end{bmatrix}\begin{bmatrix}
U_1 & U_5 & U_9 \\ 
U_2 & U_6 & U_{10} \\ 
U_3 & U_7 & U_{11} \\ 
U_4 & U_8 & U_{12} 
\end{bmatrix}. 
\]
Computations in  \textit{Macaulay2} shows that $\htt(\Kitt^\fg(3,\ff)) = \htt(\RI(3,\ff)) = 3$. Hence one obtains an specialization $\Kitt(\fa,I)$ of $\Kitt^\fg(3,\ff)$ such that $\htt(\Kitt^\fg(3,\ff)) < \htt(\Kitt(\fa,I))$. }
\end{Example}
Taking this direction, we proved that, under more restrictive conditions on the ring, the following inequality holds \begin{gather*}
        \htt(\Kitt^\fg(s,\ff)_{\fn}) \geq \htt(\Kitt(\fa,I)_{\fm})
    \end{gather*}
for any maximal ideal $\fn \subseteq S$ containing $\Kitt^\fg(s,\ff) + (\mathbf{x})$, where $\fm = \fn\cap R$, $\mathbf{x}=U_{11} - c_{11},\dots,U_{rs} - c_{rs}$ and $c_{11},\dots,c_{rs} \in R$ are such that $a_i = \sum_{k=1}^rc_{ki}f_k$ for all $i=1,\dots,s.$ 

Two lemmas are required to prove this result.
\begin{Lemma}\label{lemma345}
    Let $R$ be a commutative ring and $c_1,\dots,c_n$ elements of $R$. Consider $S = R[X_1,\dots,X_n]$ the polynomial extension of $R$ in $n$ indeterminates and let $\fn$ be an ideal of $R$ containing $\fc := (X_1 - c_1,\dots,X_n - c_n)$. Then
\begin{enumerate}[(i)]
    \item If $\fn$ is a proper ideal of $S$, then $\big(\{f(c_1,\dots,c_n)\:\:;\:\:f \in \fn\}\big)S \subset S$ also is a proper ideal of $S$;
    \item If $\fn$ is a maximal ideal of $S$, then $\fn = \fc + \big(\{f(c_1,\dots,c_n)\:\:;\:\:f \in \fn\}\big)S$.
\end{enumerate}
\end{Lemma}
\noindent\textit{Proof:} \textit{(i)} Suppose by contradiction that $\big(\{f(c_1,\dots,c_n)\:\:;\:\:f \in \fn\}\big) = S$. Let $f_1,\dots,f_m \in \fn$ and $g_1,\dots,g_m \in S$ such that \begin{gather*}
    \sum_{k=1}^mf_k(c)g_k(X) = 1.
\end{gather*}
Applying the division algorithm, there are $h_1,\dots,h_k \in \fc$ such that $f_k = h_k + f_k(c)$ for all $k=1,\dots,m$. Thus
\begin{gather*}
    1 = \sum_{k=1}^m(f_k(X) - h_k(X))g_k(X) = \sum_{k=1}^mf_k(X)g_k(X) -  \sum_{k=1}^m h_k(X)g_k(X) \in \fn + \fc = \fn,  
\end{gather*}
giving a contradiction.\\[0.3cm]
\noindent\textit{(ii)} By using the division algorithm, one can see that $\fn \subseteq \fc + \big(\{f(c_1,\dots,c_n)\:\:;\:\:f \in \fn\}\big)S$. Hence, since $\fn$ is maximal, it is enough to show that $\fc + \big(\{f(c_1,\dots,c_n)\:\:;\:\:f \in \fn\}\big)S$ is a proper ideal. Suppose by contradiction that $\fc + \big(\{f(c_1,\dots,c_n)\:\:;\:\:f \in \fn\}\big)S = S$. Let $h \in \fc$, $f_1,\dots,f_m \in \fn$ and $g_1,\dots,g_m \in S$ such that \begin{gather*}
    h(X) + \sum_{k=1}^mf_k(c)g_k(X) = 1.
\end{gather*} 
Evaluating in $(c_1,\dots,c_n)$, one gets that $\sum_{k=1}^mf_k(c)g_k(c) = 1$, contradicting the item $(i)$. Hence $\fc + \big(\{f(c_1,\dots,c_n)\:\:;\:\:f \in \fn\}\big)$ is a proper ideal of $S$. 
\qed
\begin{Lemma}\label{polynomialExtension}
Let $R$ be a Noetherian equi-codimensional ring i.e. $\dim(R) = \htt(\fm)$ for all maximal ideal $\fm$. Let $S = R[x_1,\dots,x_n]$ be the polynomial extension of $R$ in $n$ indeterminates and $c_1,\dots,c_n \in R$. Then, for any maximal ideal $\fn$ of $S$ containing the ideal $(x_1 - c_1,\dots, x_n - c_n)$, one has $$\htt(\fn) = \dim(S) = \dim(R) + n.$$
\end{Lemma}
\noindent\textit{Proof:} Let $\fn$ be a maximal ideal of $S$ containing $\mathfrak{c}:=(x_1 - c_1,,\dots, x_n - c_n)$. Firstly notice that, since $\fn$  contains $\mathfrak{c}$, then $\fn = \mathfrak{c} + (\fn\cap R)S$. Indeed, by \autoref{lemma345}, one has 
$$\fn = \fc + \big(\{f(c_1,\dots,c_n)\:\:;\:\:f(x_1,\dots,x_n) \in \fn\}\big)S =: \fc + \fa.$$
Thus $\fn\cap R = \fa\cap R$ and so 
\begin{gather*}
   \fn = \fc + (\fa\cap R)S = \fc + (\fn\cap R)S \subseteq \fn + \fn = \fn
\end{gather*}
and the equality follows. As $\fn\cap R$ is a prime ideal of $R$, then $\fn \subseteq \mathfrak{c} + \fm S$ for some maximal ideal $\fm$ of $R$ containing $\fn\cap R$. Note this inclusion actually is an equality. Indeed, if it was not, there would be $a_1,\dots,a_m\in \fm$, $g_1,\dots,g_m \in S$ and a polynomial $f \in S$ such that $f(c_1,\dots,c_n) = 0$ and 
\begin{gather*}
    f(x_1,\dots,x_n) + \sum_{k=1}^ma_kg_k = 1.
\end{gather*}
Evaluating this expression in $(c_1,\dots,c_n)$, one gets a contradiction. Thus $\fn = \mathfrak{c} + \fm S$ and $\fn \cap R = \fm$. Applying \cite[Theorem 15.1]{matsumura1989commutative}, one gets that
\begin{gather*}
    \htt_S(\fn) = \htt_R(\fm) + \dim\bigg(\frac{S_{\fn}}{\fm S_{\fn}}\bigg) = \dim(R) + \dim\bigg(\frac{S_{\fn}}{\fm S_{\fn}}\bigg). 
\end{gather*}
As $S_{\fn}/\fm S_{\fn} \cong (R/\fm)[x_1,\dots,x_n]_{\fc}$ is the localization of a polynomial ring over a field at a maximal ideal, its dimension is $n$, hence one gets
\begin{gather*}
    \dim(S) \geq \htt_S(\fn)  = \dim(R) + \dim\bigg(\frac{S_{\fn}}{\fm S_{\fn}}\bigg) = \dim(R) +n = \dim(S)
\end{gather*}
and the statement follows.
\qed\\[0.3cm]
\indent The following example shows the hypothesis that $\fn$ contains some ideal of form $\fc$ in \autoref{polynomialExtension} is  necessary. 
\begin{Example} {\normalfont
     Let $R = \mathbb{K}\llbracket x \rrbracket$ be the ring of formal series over a field $\mathbb{K}$. Consider $S = R[y]$ the polynomial extension of $R$ in one indeterminate. Notice that $R$ is a local ring. The maximal ideal $\fn= (1-xy)\subseteq S$ does not contain an element of form $y-c$ and $\htt(\fn) = 1 \neq 2 = \dim(S)$.} 
\end{Example}
\indent Finally, we are ready to establish the promised result. Recall that a ring $R$ is said to satisfy the \textit{dimension equality} if $$\dim(R) = \dim(I) +  \dim(R/I)$$ for all ideal $I$ of $R$.
\begin{Proposition}
\label{heightKitt}
    Let $R$ be Cohen-Macaulay local rings (or an  affine domain), $\fa \subset I$ ideals of $R$ and $s$ a positive integer. Consider $\mathbf{a}=\:a_1,\dots,a_s$, $\ff=\: f_1,\dots,f_r$ systems of generators of the ideals $\fa$, $I$, respectively and $c_{11},\dots,c_{rs} \in R$ such that $a_i = \sum_{k=1}^rc_{ki}f_k$. Let $S = R[U_{ij}\:\:;\:\: 1\leq i \leq r,\: 1 \leq j \leq s]$ be the polynomial extension of $R$ in $rs$ indeterminates and $\mathbf{x}=\: U_{11} - c_{11}, \dots, U_{rs} - c_{rs}$. Then, given a maximal ideal $\fn$ of $S$ containing $\Kitt^\fg(s,\ff) +(\mathbf{x})$ and denoting $\fm := \fn \cap R$, one has \begin{gather*}
        \htt(\Kitt^\fg(s,\ff)_{\fn}) \geq \htt(\Kitt(\fa,I)_{\fm}).
    \end{gather*}
In particular, if $\fa:_R I$ is an $s$-residual intersection of $I$, then $\R(s,\ff)_{{\fn}}$ is an $s$-residual intersection for all maximal ideal $\fn$ of $S$ containing $\R(s,\ff) +(\mathbf{x})$. 
\end{Proposition}
\noindent\textit{Proof:} We first notice  that $R$ and $S$ are both  locally universally catenary Noetherian  rings  such that all maximal ideals have the same height and all local rings satisfies the dimension equality. Furthermore, if the ideal $\Kitt^\fg(s,\ff) +(\mathbf{x})$ is not proper, then \autoref{Prop3.1} implies that $\Kitt(\fa,I)=(1)$. Therefore $\fa=I$ by \autoref{Main}(iii), which is not the case. 

Let $\mathfrak{a} = (a_1,\dots,a_s) \subset I$ be an ideal generated by $s$ elements and $\Phi = [c_{ij}]_{r\times s}$ such that \begin{gather*}
    \begin{bmatrix}
a_1 & a_2 & \cdots  &a_s 
\end{bmatrix} = \begin{bmatrix}
f_1 & f_2 & \cdots  & f_r 
\end{bmatrix}\begin{bmatrix}
c_{11} & c_{12}  &\cdots  & c_{1s} \\ 
c_{21} & c_{22} & \cdots & c_{2s} \\ 
\vdots & \vdots &\ddots  & \vdots \\ 
c_{r1} & c_{r2} &\cdots  & c_{rs} 
\end{bmatrix} = \begin{bmatrix}
f_1 & f_2 & \cdots  & f_r 
\end{bmatrix}\Phi. 
\end{gather*}
Now let $a'_1,\dots,a'_s \in S$ such that \begin{gather*}
    \begin{bmatrix}
a'_1 & a'_2 & \cdots  &a'_s 
\end{bmatrix} = \begin{bmatrix}
f_1 & f_2 & \cdots  & f_r 
\end{bmatrix}\begin{bmatrix}
U_{11} & U_{12}  &\cdots  & U_{1s} \\ 
U_{21} & U_{22} & \cdots & U_{2s} \\ 
\vdots & \vdots &\ddots  & \vdots \\ 
U_{r1} & U_{r2} &\cdots  & U_{rs} 
\end{bmatrix} = \begin{bmatrix}
f_1 & f_2 & \cdots  & f_r 
\end{bmatrix}\Psi. 
\end{gather*}
Let $\fn$ be a maximal ideal of $S$ which contains $\Kitt^{\mathfrak{g}}(s,\ff) + (\mathbf{x})$. Observe that $\fn\cap R = \fm$ is a maximal ideal of $R$. Furthermore, by \autoref{Prop3.1}, one has 
\begin{gather}\label{kittlocal}
   \frac{S_{\fn}}{(\Kitt^{\mathfrak{g}}(s,\ff) + (\mathbf{x}))_{\fn}} \cong \frac{R_{\fm}}{\Kitt(\fa,I)_{\fm}}.
\end{gather}
Since the catenary property is stable under localization and the rings $R_{\fm},\:S_{\fn}$ satisfy the dimension equality \cite[page 31]{matsumura1989commutative}, \cite[Lemma 18.6]{HeremannBlowing}, one has
\begin{gather*}
    \dim(R) - \htt(\Kitt(\fa,I)_{\fm}) = \dim(R_{\fm}) - \htt(\Kitt(\fa,I)_{\fm}) = \dim\bigg(\frac{R_{\fm}}{\Kitt(\fa,I)_{\fm}}\bigg) \\ = \dim\bigg(\frac{S_{\fn}}{(\Kitt^{\mathfrak{g}}(s,\ff) + (\mathbf{x}))_{\fn}}\bigg) \geq \dim\bigg(\frac{S_{\fn}}{\Kitt^{\mathfrak{g}}(s,\ff)_{\fn}}\bigg) -rs = \dim(S_{\fn}) - \htt(\Kitt^\fg(s,\ff)_{\fn}) -rs \\ = \dim(S) - \htt(\Kitt^\fg(s,\ff)_{\fn}) -rs = \dim(R) - \htt(\Kitt^\fg(s,\ff)_{\fn}),
\end{gather*}
which implies that $\htt(\Kitt^\fg(s,\ff)_{\fn}) \geq \htt(\Kitt(\fa,I)_{\fm})$. In particular, if $\fa:_RI$ is an $s$-residual intersection, one has $$\htt(\Kitt^\fg(s,\ff)_{\fn}) \geq \htt(\Kitt(\fa,I)_{\fm}) \geq \htt(\Kitt(\fa,I))  = \htt(\fa:_RI) \geq s$$ for all maximal ideal $\fn$ of $S$ containing $\Kitt^\fg(s,\ff) +(\mathbf{x})$.
\qed\\[0.3cm]
\indent Combining  \autoref{deformg+1} and \autoref{heightKitt}, one obtains the following result. 
Here we show that for the first two extreme cases of residual intersections, the generic residual is a deformation of arbitrary residual. The main point of the following corollary is that the ideal $I$ need not to satisfy any condition. 
\begin{Corollary}\label{Lastmain}
    Let $R$ be a Cohen-Macaulay local ring and $J = \mathfrak{a}:_R I$ an $s$-residual intersection of $I$. Assume $s\leq \htt(I)+1$. Let $\ff = f_1,\dots,f_r$ and $\mathbf{a}= a_1,\dots,a_s$ be systems of generators for $I$ and $\fa$, respectively, with $[\mathbf{a}]=[c_{{ij}}][\ff]$. Let $S=R[U_{ij}]$ the polynomial extension of $R$ in $rs$ indeterminates. Then, for any maximal ideal $\fn$ containing $\R(s,\ff)+(U_{ij}-c_{ij})$, one has
   \begin{center}
    $(S_{\fn},\R(s,\ff)_{\fn})$ is a deformation of $(R,J)$. 
    \end{center}
\end{Corollary}
\begin{proof} According to \autoref{Main}(iii) and \autoref{heightKitt}, one has $$\htt(\R(s,\ff)_{\fn}) = \htt(\Kitt^\fg(s,\ff)_{\fn})\geq s.$$ Thus $\R(s,\ff)_{\fn}$ is an $s$-residual intersection of $IS_{\fn}$. Since $s\leq \htt(I)+1$, \autoref{Main}(viii) implies that $J=\Kitt(\fa,I)$ and $\R(s,\ff)_{\fn}=\Kitt^{\fg}(s,\ff)_{\fn}$ for any maximal ideal $\fn$ containing $(U_{ij}-c_{ij})=:\mathbf{x}$. Applying \autoref{deformg+1}(i), one concludes the sequence $\mathbf{x}_{\fn}$ is regular over $S_{\fn}$ and $S_{\fn}/\Kitt^{\fg}(s,\ff)_{\fn}$. Hence, by \autoref{kittlocal}, one has
\begin{gather*}
  \frac{S_{\fn}/(\mathbf{x})_{\fn}}{(\R(s,\ff) + (\mathbf{x}))_{\fn}/(\mathbf{x})_{\fn}} = \frac{S_{\fn}}{(\Kitt^{\mathfrak{g}}(s,\ff) + (\mathbf{x}))_{\fn}} \cong \frac{R_{\fm}}{\Kitt(\fa,I)_{\fm}} = \frac{R}{J}.
\end{gather*}
Hence the statement follows.
\end{proof}
By applying other properties of $\Kitt$ ideal represented in  \autoref{Main}, one may find condition under which the generic residual is a deformation of arbitrary one.
\begin{Corollary}\label{Lastmain2}
    Let $R$ be a Cohen-Macaulay local ring and $J = \mathfrak{a}:_R I$ an $s$-residual intersection of $I$.
    Let $\ff = f_1,\dots,f_r$ and $\mathbf{a}= a_1,\dots,a_s$ be systems of generators for $I$ and $\fa$, respectively, with $[\mathbf{a}]=[c_{{ij}}][\ff]$. 
    Suppose that $I$ is of height $g$ and satisfies any of the following conditions
    \begin{itemize}
    \item  $I$ satisfies the $G_{s}$ condition, it  is weakly $(s-2)-$residually $S_{2}$ and   satisfies SDC$_1$ condition at level $\min\{s-g-2,r-g\}$;
    \item   $I$ satisfies $\SD_1$;
    \end{itemize}
    Let $S=R[U_{ij}]$ the polynomial extension of $R$ in $rs$ indeterminates. Then, for any maximal ideal $\fn$ containing $\R(s,\ff)+(U_{ij}-c_{ij})$, one has
   \begin{center}
    $(S_{\fn},\R(s,\ff)_{\fn})$ is a deformation of $(R,J)$. 
    \end{center}
\end{Corollary}

 \vspace{.7cm}
{\bf Acknowledgement.}  The authors would like to thank Vaibhav Pandey for useful discussions 
on the arithmetic rank of residual intersections.

\bibliographystyle{alpha}
\bibliography{References}

\end{document}